\documentclass[11pt]{amsart}

\usepackage{amsmath,amssymb,amsthm,url,mathtools,microtype}
\usepackage[margin=1.2in]{geometry}

\author{Paul Pollack}
\address{Department of Mathematics\\University of Georgia\\Athens, GA 30602}
\email{pollack@uga.edu}

\author{Carl Pomerance}
\address{Department of Mathematics\\Dartmouth College\\Hanover, NH 03755}
\email{carlp@math.dartmouth.edu}

\author{Lola Thompson}
\address{Department of Mathematics\\Oberlin College\\Oberlin, OH 44074}
\email{lola.thompson@oberlin.edu}

\title{Divisor-sum fibers}

\DeclareMathAlphabet{\curly}{U}{rsfs}{m}{n}

\newtheorem{thm}{Theorem}[section]

\newtheorem{prop}[thm]{Proposition}

\newtheorem{conjecture}[thm]{Conjecture}

\newtheorem*{theorem*}{Theorem}
\newtheorem{hyp}[thm]{Hypothesis}

\theoremstyle{remark}
\newtheorem*{remark}{Remark}

\numberwithin{equation}{section}

\newcommand\A{\curly{A}}

\newcommand\E{\curly{E}}
\newcommand\rad{\mathrm{rad}}

\renewcommand{\phi}{\varphi}
\makeatletter
\renewcommand{\pod}[1]{\mathchoice
  {\allowbreak \if@display \mkern 18mu\else \mkern 8mu\fi (#1)}
  {\allowbreak \if@display \mkern 18mu\else \mkern 8mu\fi (#1)}
  {\mkern4mu(#1)}
  {\mkern4mu(#1)}
}

\begin{document}

\begin{abstract} Let $s(\cdot)$ denote the sum-of-proper-divisors function, that is, $s(n) = \sum_{d\mid n,~d<n}d$. Erd\H{o}s--Granville--Pomerance--Spiro conjectured that for any set $\A$ of asymptotic density zero, the preimage set $s^{-1}(\A)$ also has density zero. We prove a weak form of this conjecture: If $\epsilon(x)$ is any function tending to $0$ as $x\to\infty$, and $\A$ is a set of integers  of cardinality at most $x^{\frac12+\epsilon(x)}$, then the number of integers $n\le x$ with $s(n) \in \A$ is $o(x)$, as $x\to\infty$. In particular, the EGPS conjecture holds for infinite sets with counting function $O(x^{\frac12 + \epsilon(x)})$. We also disprove a hypothesis from the same paper of EGPS by showing that for any positive numbers $\alpha$ and $\epsilon$, there are integers $n$ with arbitrarily many $s$-preimages lying between $\alpha(1-\epsilon)n$ and $\alpha(1+\epsilon)n$. Finally, we make some remarks on solutions $n$ to congruences of the form $\sigma(n) \equiv a\pmod{n}$, proposing a modification of a conjecture appearing in recent work of the first two authors. We also improve a previous upper bound for the number of solutions $n \leq x$, making it uniform in $a$.
\end{abstract}

\maketitle


\section{Introduction}
Let $s(n)$ denote the sum-of-proper-divisors of an integer $n$, i.e., $s(n) = \sigma(n) - n$. The function $s(\cdot)$ has been a source of fascination since the time of the ancient Greeks, who classified numbers as \textit{perfect}, \textit{abundant}, or \textit{deficient}, according to whether $s(n) = n$, $s(n) > n$, or $s(n) < n$. Two thousand years later, the desire to understand statistical properties of  $s(n)$ played a motivating role in the early development of probabilistic number theory by figures such as Schoenberg, Davenport, Erd\H{o}s, and Wintner.

It is interesting that the function $s$ can map sets of asymptotic density 0 to sets of positive density.
Indeed, if $\A$ is the set of numbers $pq$, where $p,q$ are primes, then $\A$ has zero asymptotic density, yet
$s(\A)$ has asymptotic density $1/2$, the latter claim coming from the fact that the Goldbach conjecture
has at most a zero-density set of exceptions.  We also know of sets $\A$ of positive density such that
$s^{-1}(\A)$ not only has zero density, but, in fact, is empty (see Erd\H{o}s \cite{erdos}).

Our focus in this paper is on what can be said about $s^{-1}(\A)$ when $\A$ has asymptotic density 0.
Part of the landscape here is a conjecture proposed by Erd\H{o}s, Granville, Pomerance, and Spiro \cite{egps} (hereafter ``EGPS'') in their study of the iterates of $s(n)$.

\begin{conjecture}\label{conj:EGPS} Let $\A$ be a set of asymptotic density zero. Then $s^{-1}(\A)$ also has asymptotic density zero.
\end{conjecture}

If true, a consequence (see \cite{egps}) would be that for each fixed positive integer $k$, but for
a set of numbers $n$ of asymptotic density 0, if $s(n)<n$, then $s_k(n)<s_{k-1}(n)<\dots<n$, where
$s_j$ is the $j$-fold iteration of $s$.
With the inequality signs reversed, this is an unconditional theorem of Erd\H os. Another consequence of Conjecture \ref{conj:EGPS} is given in \cite{firstfunction}.

Some special cases of the EGPS conjecture can be read out of the literature. For  example, it is proved in \cite{pollack14} that if $\A$ is the set of primes, then the counting function of $s^{-1}(\A)$ is $O(x/\log x)$. Also, Troupe showed in \cite{troupe} that $s^{-1}(\A_\epsilon)$ has density zero for each of the sets
\[ \A_\epsilon = \{m : |\omega(m) - \log \log m| > \epsilon \log \log m\}. \]
In \cite{pollack15}, it is shown that $s^{-1}(\A)$ has density zero when $\A$ is the set of palindromes (in any given base). All of these arguments make critical use of structural features of $\A$; the methods do not carry over to arbitrary sets with similar counting functions.


In the present paper, we make some partial progress on Conjecture \ref{conj:EGPS}. In contrast with the aforementioned results, the structure of $\A$ is irrelevant in the theorem, but we must assume a strong condition on the count of elements of $\A$.

\begin{thm}\label{thm:preimageofs} Let $\epsilon = \epsilon(x)$ be a fixed function tending to $0$ as $x \rightarrow \infty$. Suppose that $\A$ is a set of at most $x^{1/2+\epsilon(x)}$ positive integers. Then, as $x\to\infty$,
$$\#\{n\le x: s(n) \in \A\} = o_\epsilon(x), $$
uniformly in the choice of $\A$.
\end{thm}
\noindent As an example, since $s(n)<2n\log\log n$ for all large $n$ and since there are only
$O_b(\sqrt{x\log\log x})$ base-$b$ palindromes up to $2x\log\log x$, Theorem \ref{thm:preimageofs} implies the theorem of \cite{pollack15} alluded to above.

Our proof of Theorem \ref{thm:preimageofs} is presented in \S \ref{sec:sparse}. We borrow some ideas from recent work of Booker \cite{booker}; Booker's arguments by themselves  almost immediately give the slightly weaker result with $x^{1/2+\epsilon(x)}$ replaced by $x^{1/2 - \epsilon}$ for any fixed (constant) $\epsilon > 0$.

EGPS \cite[p.\ 170]{egps} point out that their Conjecture \ref{conj:EGPS} would be a consequence of the following assertion about the sizes of elements in a fiber.
\begin{hyp}
\label{assert:quad}
For each positive number $\theta$ there exists a constant $C_\theta$ such that for all positive integers $m$ there exist at most $C_\theta$ numbers $n \leq \theta m$ with $s(n) = m.$
\end{hyp}
\noindent The authors write in \cite{egps}: ``We are not sure we believe this hypothesis and in fact it may be possible
to disprove it." Our second theorem, shown in \S\ref{sec:assertion}, disproves Hypothesis \ref{assert:quad} in a strong way.

\begin{thm}\label{thm:2ndmain} There is a constant $c>0$ for which the following holds. Let $\alpha$ and $\epsilon$ be positive real numbers. There are infinitely many $m$ with at least $\exp(c\log m/\log\log m)$ $s$-preimages 
that lie in the interval $(\alpha(1-\epsilon)m, \alpha(1+\epsilon)m)$.
\end{thm}

\noindent Our proof of Theorem \ref{thm:2ndmain} shows that $c=1/7$ is admissible.

In \S \ref{sec:tamsconjecture}, we use ideas presented in the previous sections to study solutions to equations of the form $\sigma(n) = kn+a$, where $k$ and $a$ are integers and $k \geq 0$. Such equations have been studied extensively by the second author over the course of his career, dating back to a 1975 paper \cite{pom75}, and are connected to many classical problems in number theory. We show that a conjecture in \cite{app}, that was modified in \cite{polpom}, needs to be further modified. Furthermore, we show that an upper bound for the count of $n \leq x$ satisfying $\sigma(n)\equiv a\pmod n$, which is given in \cite{pom75}, can be made uniform in $a$. In particular, we prove the following.

\begin{thm}
\label{prop:reallyfinalnotkidding}
For any integer $a$, the number of $n\le x$ with $\sigma(n)\equiv a\pmod n$ is $O(x/\log x)$, uniformly in $a$.
\end{thm}

We remark that Corollary 3 in \cite{pom75} appears to give a uniform upper bound, but the dependence on $a$ is suppressed in the notation.

\subsection*{Notation} Throughout this paper, $n$ and $m$ denote positive integers and $\ell$, $p$, and $q$ are primes. We let $P(n)$ and $P^{-}(n)$ denote the largest and smallest prime factors of $n$, respectively, and we let $\rad(n) = \prod_{p \mid n} p$. We write $\log_k{x}$ for the $k$th iterate of the natural logarithm function at $x$.

\section{The preimages of very sparse sets: Proof of Theorem \ref{thm:preimageofs}}
\label{sec:sparse}





By replacing $\epsilon(x)$ with $\max\{\epsilon(x),1/\log\log{x}\}$, we can assume that $x^{\epsilon(x)} \geq x^{1/\log\log{x}}$. We start by introducing a set $\E$ of ``exceptional preimages.'' We let $\E$ be the set of $n \leq x$ such that at least one of the following holds:
\begin{enumerate}
\item[(a)] $n$ has no prime factor in $(1,\log{x}]$,
\item[(b)] $n$ has a divisor in $(x^{1/2-10\epsilon(x)},x^{1/2+10\epsilon(x)})$,
\item[(c)] $n$ has a squarefull part exceeding $x^{2\epsilon(x)}$,
\item[(d)] $n \leq \sqrt{x}$.
\end{enumerate}

By a simple sieve (inclusion-exclusion), there are $\ll x/\log \log x$ numbers $n \leq x$ satisfying (a). Known results on the distribution of divisors (see, e.g., the sharp result of Ford \cite[Theorem 1]{ford08}) imply that the number of $n \leq x$ satisfying (b) is $o(x)$. The number of $n \leq x$ satisfying (c) is $\ll x^{1-\epsilon(x)}$, which is also $o(x)$. Finally, the number of $n$ satisfying (d) is trivially $O(x^{1/2})$. We conclude that $\#\E = o(x)$. Thus, for the sake of proving the theorem, it suffices to bound the number of non-exceptional $n$ with $s(n) \in \A$.

We will show that for each $a \in \A$, the number of non-exceptional preimages of $a$ is $\ll x^{1/2-9\epsilon(x)}$. Since $\#\A \le x^{1/2+\epsilon(x)}$, the theorem follows.

Let $a \in \A$, and let $n$ be a non-exceptional preimage of $a$.
Write $n = de$, where $d$ is the largest divisor of $n$ not exceeding $\sqrt{x}$.
Since $n$ is non-exceptional, $d \leq x^{1/2-10\epsilon(x)}$.

Since $n > \sqrt{x}$, we have $e > 1$.
Let $p$ be any prime divisor of $e$. Then $dp \mid n$. By the choice of $d$, we have $dp > \sqrt{x}$, and since $n$ is non-exceptional, \begin{equation}\label{eq:dpbound} dp \geq x^{1/2+10\epsilon(x)}.\end{equation} Thus,
\begin{equation}\label{eq:smallpebound} p \geq x^{1/2+10\epsilon(x)}/d \geq x^{20\epsilon(x)}.\end{equation} If $p$ divides both $d$ and $e$, then $n$ has squarefull part at least $p^2 \geq x^{40\epsilon(x)}$, contradicting that $n$ is non-exceptional. Hence, $\gcd(d,e)=1$.

Observe that \begin{align*} \frac{\sigma(e)}{e} &= \prod_{p^k || e} \left(1+ \frac{1}{p} + \cdots + \frac{1}{p^k}\right) \leq \prod_{p | e} \left(1+\frac{1}{p-1}\right) \\
            &\leq \exp\left(\sum_{p|e} \frac{1}{p-1}\right) \leq \exp\left(\frac{\omega(e)}{P^{-}(e)-1}\right) \leq \exp\left(\frac{2\omega(e)}{P^{-}(e)}\right),\end{align*}
where $\omega(e)$ is the number of distinct primes dividing $e$.
But $\omega(e) \leq \log{e} / \log{2} \leq \log{x}/\log{2}$ while, from \eqref{eq:smallpebound}, $P^{-}(e) \geq x^{20\epsilon(x)} \gg (\log x)^2$. So,
$\exp\left(\frac{2\omega(e)}{P^{-}(e)}\right) =1 + O\left(\frac{\log x}{P^{-}(e)}\right)$, and hence
$$\frac{s(e)}{e} = \frac{\sigma(e)}{e}-1  \ll \frac{\log{x}}{P^{-}(e)}.$$
Keeping in mind that $de = n\le x$, we deduce that
\begin{equation}\label{eq:seinequality} s(e) \ll e \frac{\log{x}}{P^{-}(e)} \le \frac{x \log{x}}{d\cdot P^{-}(e)} \le x^{1/2-10\epsilon(x)} \log{x}; \end{equation}
here the last inequality follows from \eqref{eq:dpbound}.

Since $\gcd(d,e)=1$, we also have that
\begin{align}\label{**} a = s(de) = \sigma(d) s(e) + s(d) e.\end{align}
So if $g = \gcd(\sigma(d),s(d))$, then $g \mid a$, and
$a/g = (\sigma(d)/g) s(e) + (s(d)/g) e.$
Hence,
$$s(e) (\sigma(d)/g) \equiv a/g \pmod{s(d)/g}.$$
Since $\sigma(d)/g$ and $s(d)/g$ are relatively prime, given $d$ we see that $s(e)$ lies in a uniquely determined residue class modulo $s(d)/g$. Combined with \eqref{eq:seinequality}, we deduce that the number of possibilities for $s(e)$ is $\ll 1 + x^{1/2-10\epsilon(x)} \log x\cdot g/s(d)$. Moreover, $s(d) \geq d/P^{-}(d) \ge d/\log x$. So the number of choices for $s(e)$, given $d$, is \[ \ll 1 + x^{1/2-10\epsilon(x)} (\log x)^2g/d.\] 
 From \eqref{**}, $d$ and $s(e)$ determine $e$, and hence this last displayed
quantity is also a bound on the number of possibilities for $n=de$, given $d$. 

We now sum on possible values of $g$ and $d$. Each $d$ under consideration has the form $gh$, where $h\leq x^{1/2-10\epsilon(x)}/g.$ Thinking of $g$ as fixed and summing on $d=gh$ gives a bound of
$\ll (\log{x})^3 x^{1/2-10\epsilon(x)}.$
Summing on the $\tau(a)$ divisors $g$ of $a$ bounds the number of possibilities for $n$ by $\ll \tau(a) (\log{x})^3 x^{1/2-10\epsilon(x)}.$
By the maximal order of the divisor function, $\tau(a) \leq x^{0.7/\log\log{x}}$ for large $x$. Since $x^{0.7/\log\log{x}} (\log{x})^3 < x^{1/\log\log{x}} \leq x^{\epsilon(x)}$ for large $x$, we see that the number of $n$ that arise in this way is
$\ll x^{1/2-9\epsilon(x)},$ as claimed.

\section{Disproof of Hypothesis \ref{assert:quad}: Proof of Theorem \ref{thm:2ndmain}}
\label{sec:assertion}

The following can be deduced from \cite[Theorem 2.1]{agp}.

\begin{thm}\label{thm:agp} For each $\epsilon>0$ and number $x$ sufficiently large depending on $\epsilon$,
there is a finite set $\{m_1,m_2,\dots,m_t\}$ of integers, where $t$ depends only on $\epsilon$ and
each $m_i>\log x$, with the following property.
 If $m \leq x^{5/12 - \epsilon}$  and  $m$ is not divisible by any of $m_1,...,m_{t}$, then for each integer
 $u$ coprime to $m$,  there are $\gg_\epsilon \frac{x}{\varphi(m) \log x}$ primes $p\le x$ with
 $p\equiv u \pmod{m}$. \end{thm}

We prove the following result, which generalizes ideas of Prachar \cite{prachar} and
Erd\H os \cite{E36}.

\begin{thm}\label{thm:generalprachar} There is a positive absolute constant $c$ such that, for all pairs of integers $a, b$ with $a\ne0,b>0$, there are infinitely many integers $k$ with more than $\exp(c \log k/\log \log k)$ representations as $(bp+a)(bq+a)$ with $p, q$ primes. \end{thm}

\begin{proof}[Proof of Theorem \ref{thm:generalprachar}] We fix a choice for $a, b$ and let $v=\max\{|a|,b\}$.
We choose $x,\epsilon$ in Theorem \ref{thm:agp} with $x> e^v$ and $\epsilon=\frac1{13}$.
Let $\E$ be the smallest set of primes such that each exceptional $m_i$ from Theorem \ref{thm:agp} has a prime factor in $\E$.
Let $M$ be the product of all odd primes $\leq \frac{1}{3} \log x$ that are not in $\E$ and do not divide $ab$.
Since $x > e^{v}$, it follows that $|a|, b< \log x$, so the number of prime factors of $ab$ is $\ll \log \log x$.
We have $\omega(M) \sim \frac{1}{3} \log x/\log \log x$ and $M = x^{1/3 + o(1)}.$
For each $d \mid M$, let $r_d$ be the solution to $br_d + a \equiv 0 \pmod{d}$.
Since $a, b$ are coprime to $d$, we must have $r_d$ coprime to $d$ as well.
Consider the primes $p \leq x$ with $bp+a \equiv 0 \pmod{d}.$  These are the primes
 $p \equiv r_d \pmod{d}$.  By Theorem \ref{thm:agp} there are $\gg \frac{x}{\varphi(d) \log x}$ such $p$. 

Now also look at primes $q \leq x$ with $(bq+a, M) = \frac{M}{d}.$ In order that $r$ mod $M$ be a coprime residue class with $(br+a,M)=\frac{M}{d}$, it is necessary and sufficient that
\[ r \equiv -a/b\pmod{\ell} \quad\text{for all $\ell \mid \frac{M}{d}$}, \quad\text{while}\quad r\not\equiv 0,-a/b \pmod{\ell}\quad\text{for all $\ell \mid d$}. \]
There are precisely $\prod_{\ell \mid d} (\ell-2)$ such residue classes $r\bmod{M}$, and so the number of primes $q\le x$ belonging to one of these classes is
\[ \gg \frac{x}{\log x} \frac{\prod_{\ell \mid d} (\ell-2)}{\phi(M)}. \]
The number of pairs $p, q$ is therefore $$\gg \frac{x}{\varphi(d) \log x} \cdot \frac{x \prod_{\ell \mid d}(\ell-2)}{\varphi(M) \log x} = \frac{x^2}{\log^2 x}\cdot  \frac{1}{\phi(M)} \prod_{\ell \mid d} \frac{\ell-2}{\ell-1} \gg \frac{x^2}{M \log^2 x}.$$ Because of the condition $(bq+a, M) = \frac{M}{d}$, different values of $d$ correspond to different values of $q$. Thus, we must have $$\gg \frac{x^2 \tau(M)}{M \log^2 x}$$ pairs $p, q$ running over all $d$'s. Map each pair $p, q$ to $(bp+a)(bq+a) \ll v^2 x^2.$ The number of integers $\ll v^2 x^2$ divisible by $M$ is $\ll \frac{v^2 x^2}{M}$. By the Pigeonhole Principle, there exists some $k \ll v^2 x^2$ with $\gg \frac{x^2 \tau(M)}{M \log^2 x}/\frac{v^2 x^2}{M}$ representations as $(bp+a)(bq+a),$ which is
 \begin{align*}
  \gg \frac{\tau(M)}{v^2 \log^2 x} &> \frac{\tau(M)}{\log^4 x}= \frac{2^{\omega(M)}}{\log^4 x} = \frac{2^{(\frac{1}{3} + o(1))(\log x/ \log \log x)}}{\log^4 x}
  = 2^{(\frac{1}{3} + o(1)) \log x/\log \log x} \\
  &= 2^{(\frac{1}{6} + o(1)) \log(v^2 x^2)/\log \log (v^2 x^2)} \geq 2^{(\frac{1}{6} + o(1)) \log k/\log \log k}.
 \end{align*}
 Since $x$ can be arbitrarily large, this argument produces infinitely many integers $k$ with at least
 $\exp(c\log k/\log\log k)$ representations as $(bp+a)(bp+q)$ with $p,q$ primes, where $c=1/10$.
 \end{proof}

\begin{remark} The best we can do using Theorem \ref{thm:agp} would be to replace $\frac{1}{3}$ with any
number smaller than $\frac{5}{12}$, which gives the result for any number $c<(5/24)\log 2$.  In particular, $c=1/7$ works. \end{remark}

\begin{proof}[Proof of Theorem \ref{thm:2ndmain}] We may assume that $0 < \epsilon < 1$. It is well-known that the values $s(n)/n$ are dense in $(0,\infty)$. Thus, we may fix an integer $n_0>1$ with \[ s(n_0)/n_0 \in \left(\alpha^{-1}\left(1-\frac{1}{2}\epsilon\right),\alpha^{-1}\left(1+\frac{1}{2}\epsilon\right)\right).\] If $p$ and $q$ are distinct primes not dividing $n_0$, then
\begin{align*} s(n_0 pq) &= \sigma(n_0)(p+1)(q+1) - n_0 pq \\
&=s(n_0) pq + \sigma(n_0) (p+q+1), \end{align*}
so that
\begin{equation}\label{eq:analogue} s(n_0) s(n_0 pq) = (s(n_0)p + \sigma(n_0))(s(n_0)q + \sigma(n_0)) +s(n_0)\sigma(n_0)-\sigma(n_0)^2.\end{equation}
By Theorem \ref{thm:generalprachar}, there are infinitely many integers $k$ having more than $\exp(c \log k/\log \log k)$ representations in the form \[ k = (s(n_0)p+\sigma(n_0))(s(n_0)q+\sigma(n_0)),\] with $p,q$ distinct and
\[ p, q > \max\{n_0,\sigma(n_0)/n_0 \cdot (s(n_0)/n_0\cdot \epsilon/12)^{-1}\}.\] (These latter conditions on $p$ and $q$ exclude only $O(1)$ representations of $k$.) Letting $k$ be a large integer of this kind, define $m$ in terms of $k$ by
\[ m = \frac{k + s(n_0)\sigma(n_0)-\sigma(n_0)^2}{s(n_0)}.\] Then $m < k$ and $m$ has at least $\exp(c\log m/\log \log m)$ representations in the form $s(n_0 pq)$. Moreover,
\[ m = s(n_0 pq) \ge s(n_0) pq \ge \left(n_0 \alpha^{-1}\left(1-\frac{1}{2}\epsilon\right)\right) pq = n_0 pq \cdot \alpha^{-1}\left(1-\frac{1}{2}\epsilon\right). \]
Thus,
\[ n_0 pq \le \alpha m \left(1-\frac{1}{2}\epsilon\right)^{-1} < (1+\epsilon)\alpha m. \]
On the other hand, using the bounds on $p, q$,
\begin{align*} \frac{s(n_0 pq)}{n_0 pq} &= \frac{s(n_0)}{n_0} + \frac{\sigma(n_0)}{n_0}\left(\frac{1}{p} + \frac{1}{q} + \frac{1}{pq}\right) \\
&< \frac{s(n_0)}{n_0} + \frac{s(n_0)}{n_0} \cdot \frac{\epsilon}{4} < \alpha^{-1} \left(1+\frac{7}{8}\epsilon\right). \end{align*}
Rearranging,
\[ n_0 pq > \alpha \cdot s(n_0 pq) (1+7\epsilon/8)^{-1} > (1-\epsilon)\alpha m.\]
Thus, $m$ has at least $\exp(c\log m/\log \log m)$ preimages $n=n_0 pq$ in $((1-\epsilon)\alpha m, (1+\epsilon)\alpha m)$.
 \end{proof}

\begin{remark} In \cite{firstfunction}, the second author writes ``it seems difficult to show there are infinitely many even $n$ with $\#s^{-1}(n) \ge 3$.'' This can now be shown. Following the above construction with $n_0=2$, one obtains infinitely many even $m$ with more than $\exp(c\log m/\log \log m)$ $s$-preimages of the form $2pq$.
\end{remark}

\section{Solutions to $\sigma(n) = kn + a$}\label{sec:tamsconjecture} Solving $s(n)=a$ is of course the same as solving $\sigma(n)=n+a$. In \cite{pom75}, \cite{app}, and \cite{polpom0} the authors study the more general equation
\begin{equation}\label{eq:appequation} \sigma(n)=kn+a,\end{equation}
for given $k$ and $a$. In this context, an integer $n$ is called a \textit{regular} solution to \eqref{eq:appequation} if
$$n = pm, \ \mathrm{where} \ p \nmid m, \ \sigma(m)/m=k, \ \mathrm{and} \ \sigma(m) = a.$$
All other solutions are called \textit{sporadic}. If there are any regular solutions at all, then there are $\gg_{a,k} x/\log x$ regular solutions up to $x$, for large $x$. In Theorem 1 of \cite{app}, it is shown that the count of sporadic solutions is much smaller: For any integer $k\ge 0$ and any integer $a$ with $|a| \le x^{1/4}$, the count of sporadic solutions up to $x$ is at most $x^{1/2+o(1)}$, as $x\to\infty$, uniformly in $k$ and $a$.

The authors of \cite{app} claim it is plausible that the upper bound $x^{1/2+o(1)}$ can be replaced with a bounded power of $\log x$, even in the wider range $|a| \leq x/2$. However, this is provably false in the case $k=0$, 
as shown by the method in Erd\H os \cite{E35}.  And it is provably false in the case $k=1$ since there are
infinitely many positive integers $a$ with $\gg a/\log^2a$ representations as $p+q+1$ where
$p,q$ are unequal primes (so that $\sigma(pq)=pq+a$).
The first two authors of the present paper
go on to state this as a conjecture in \cite[Conjecture 2.4]{polpom}, taking into account that $k$ should not be 0 or 1.

\begin{conjecture}\label{conj:tams} Let $k \geq 2$.  Let $x \geq 3$ and $a \in \mathbb{Z}$ with $|a| \leq \frac{x}{2}$. The number of solutions $n \leq x$ to $\sigma(n) = kn+a$ is $\ll (\log x)^C$, where both the implied constant and $C$ are absolute constants. \end{conjecture}

The statement of Conjecture \ref{conj:tams} in \cite{polpom} inadvertently omits the word ``sporadic,'' which is still needed. However, it turns out that confining attention to sporadic solutions is still not enough. For example, when $k = 3$, there are values of $a$ for which the equation $\sigma(n) - 3n = a$ has at least $x^{{1/2} + o(1)}$ solutions with $n \leq x$ of the form $n = 120pq$. This construction depends on the fact that 120 is 3-perfect, i.e., $\sigma(120)=3\cdot 120$. One might attempt to further salvage Conjecture \ref{conj:tams} by barring constructions like $\sigma(n) - 3n = a$ with $n = 120pq$. However, even with this modification, the upper bound on the number of solutions to $\sigma(n) = kn+a$ given in Conjecture \ref{conj:tams} is too small.

\begin{prop} Let $k$ be a positive integer, and let $C > 0$. There are infinitely many positive integers $a$ for which the equation $\sigma(n)=kn+a$ has more than $(\log{a})^{C}$ solutions $n \le a$ not of the form $m p$ or $m pq$ for any $k$-perfect number $m$.
\end{prop}

\noindent By considering values of $x=2a$, we see that Conjecture 4.1 fails for every $k$, even after restricting $n$ to sporadic solutions not of the form $mpq$ for a $k$-perfect number $m$.

\begin{proof}[Proof (sketch)] Fix a positive integer $n_0$ with $\sigma(n_0)/n_0 > 2k$. We consider numbers $n = n_0 pq$, where $p, q$ are distinct primes not dividing $n_0$. Then, writing $T=\sigma(n_0)-kn_0$, one finds that \begin{equation}\label{eq:analogue2} T(\sigma(n) - kn) = (Tp + \sigma(n_0))(Tq + \sigma(n_0)) + T\sigma(n_0) - \sigma(n_0)^2. \end{equation}
(This is the analogue of \eqref{eq:analogue}.) Following the proof of Theorem \ref{thm:2ndmain}, we produce infinitely many positive integers $a$ that can be written as $\sigma(n)-kn$ for more than $\exp(c\log a/\log \log a)$ numbers $n=n_0 pq$. All of our $n=n_0 pq$ satisfy $\sigma(n)/n  \ge \sigma(n_0)/n_0 > 2k$, so that $n <  (\sigma(n)-kn)/k = a/k \le a$. Moreover, $n$ does not have the form $mpq$ for a for a $k$-perfect number $m$. Indeed, for any primes $p$ and $q$, both $\sigma(mp)/mp$ and $\sigma(mpq)/mpq$ are bounded above by $(\sigma(m)/m) (1+1/2)(1+1/3)= 2k$, whereas $\sigma(n)/n > 2k$. Since $\exp(c\log a/\log \log a) > (\log a)^{C}$ for large $a$, the proposition follows.
\end{proof}
 Thus, we cannot hope to get the bound of $(\log x)^C$ stated in Conjecture \ref{conj:tams}. Instead, one might make the following conjecture:

\begin{conjecture}\label{conj:revisedtams} Let $k$ be a positive integer. Let $x \geq 3$ and $a \in \mathbb{Z}$ with $|a| \leq x$. The number of sporadic solutions $n \leq x$ to $\sigma(n) = kn+a$ is at most $x^{1/2 + o(1)},$
as $x\to\infty$, uniformly in $k,a$. \end{conjecture}

We have the following variant. 

\begin{thm}\label{thm:finalprop} Let $k$ be a positive integer. Let $x \geq 3$ and $a \in \mathbb{Z}$. The number of sporadic solutions $n  \le x$ to $\sigma(n) = kn + a$ is at most $x^{3/5 + o_k(1)}$, as $x \rightarrow \infty$, uniformly in $a$.\end{thm}

Note that the conclusion here is slightly weaker than that of \cite[Theorem 1]{app}: 
we have $x^{3/5+o(1)}$ in place of $x^{1/2+o(1)}$, and the result is not claimed to be uniform in $k$. The advantage of Theorem \ref{thm:finalprop} is that $a$ is no longer restricted to satisfy $|a| \le x^{1/4}$.  However, for all solutions $n\le x$,
we have $|a|=|\sigma(n)-kn|\ll_k x\log\log x$.

\begin{proof} We can assume $a\ne 0$; if $a=0$, we are counting multiply perfect numbers, and Hornfeck and Wirsing \cite{hw57} have shown that the count of such $n\le x$ is $O(x^{\epsilon})$ for any fixed $\epsilon > 0$.

If $n=p$ is a prime solution to $\sigma(n)=kn+a$, then $p(1-k) = a-1$. If $k\ge 2$, this equation determines $p$, and so there is at most one solution with $n$ prime. If $k=1$, then for there to be any solution we must also have $a=1$; but then $n=p$ is a regular solution to $\sigma(n)=kn+a$, not a sporadic solution.

Thus, we may restrict attention to solutions $n$ with at least two prime factors, counted with multiplicity. We write each such $n$ in the form
\[ n = n_0 pq, \quad\text{where}\quad q = P(n), \quad p = P(n/q). \]

We now show that the cases $n_0 \le x^{3/5}$ contribute at most $x^{3/5+o(1)}$ solutions.  To start with, we  suppose that $(pq, n_0) = 1$ and $p \neq q$.

Let $T = \sigma(n_0)-kn_0$. If $T=0$, then
\[ a = \sigma(n)-kn = \sigma(n_0)(p+1)(q+1) - kn_0 pq = kn_0 (p+q+1). \]
Thus, $n_0$ and $p+q+1$ are divisors of $a$. Since $|a|\ll x\log\log x$, there are only $x^{o(1)}$ divisors of $a$, 
and so there are only $x^{o(1)}$ possibilities for $n_0$, as well as only $x^{o(1)}$ possibilities for $q$, given $p$. Since $p \le x^{1/2}$, the case when $T=0$ leads to only $x^{1/2+o(1)}$ possibilities for $n=n_0 pq$.

Now suppose that $T\ne 0$. Then (cf. \eqref{eq:analogue2}),
\[ Ta - T\sigma(n_0) +\sigma(n_0)^2 = (Tp+\sigma(n_0)) (Tq+\sigma(n_0)). \]
If the left-hand side is nonvanishing, appealing once again to the maximal order of the divisor function we may conclude that there are only $x^{o(1)}$ possibilities for $p,q$, and so only $x^{3/5+o(1)}$ possibilities for $n=n_0 pq$.  If the left-hand side is $0$, then $\sigma(n_0)=-Tp$ or $\sigma(n_0)=-Tq$. If $\sigma(n_0)=-Tp$, then
\begin{align*} \sigma(n_0p) &= -Tp(p+1) = k n_0 p(p+1) - \sigma(n_0) p(p+1)\\
&=  k n_0 p(p+1) - p \sigma(n_0 p);
\end{align*}
hence, $\sigma(n_0 p) = k \cdot n_0 p$. Since
\[ k n_0 p q + a = \sigma(n_0 pq) = \sigma(n_0 p) (q+1) = kn_0 p(q+1) = k n_0 pq + k n_0 p, \]
we have
\[ a = k n_0 p = \sigma(n_0 p). \]
Therefore $n = (n_0 p)q$ is a regular solution to $\sigma(n)=kn+a$ and so should not be counted here. A similar analysis applies when $\sigma(n_0)=-Tq$.

It still remains to deal with the cases when either $p$ or $q$ divides $n_0$ or $p=q$. If $q$ divides $n_0$, then $p = q$. Since there are only $O(\log x)$ primes dividing $n_0$, there are only $O(x^{3/5} \log x)$ possibilities for $n=n_0 pq$. So we may suppose that $q\nmid n_0$. When $p=q$ (and $q\nmid n_0$),
\begin{align*} a &= \sigma(n_0 p^2) - k n_0 p^2 \\&= (\sigma(n_0) - k n_0) p^2 + \sigma(n_0) (p+1). \end{align*}
Given $n_0$, the final right-hand side is a polynomial in $p$ of degree $1$ or $2$, and so $p$ is determined in at most two ways, yielding $O(x^{3/5})$ possibilities for $n$. We are left with the case when $p, q$ are distinct primes, $q\nmid n_0$, but $p\mid n_0$. In this case, $p$ is determined in at most $O(\log x)$ ways from $n_0$. Moreover,
\[ a = \sigma(n_0 p q) - k n_0 pq = \sigma(n_0 p) (q+1) - k n_0 pq = q (\sigma(n_0 p)-k n_0 p) + \sigma(n_0 p). \]
Thus, as long as $\sigma(n_0 p) \ne k n_0 p$, we see $q$ is determined by $n_0$ and $p$, and hence $n = n_0 pq$ is determined in $O(x^{3/5} \log x)$ ways. But if $\sigma(n_0p) = k n_0p$, it is easy to see from the last display that $n=(n_0 p)q$ is actually a regular solution rather than a sporadic solution.

So we have shown that those cases where $n_0 \le x^{3/5}$ contribute at most $x^{3/5+o(1)}$ solutions. So we may assume that $n_0 > x^{3/5}$ and thus that $$ pq = n/n_0 \le x/n_0 < x^{2/5}.$$ Hence, $p < x^{1/5}$, and there is a divisor $d_0$ of $n$ in the interval $(x^{2/5},x^{3/5}]$. Indeed, we may choose $d_0$ as a certain product of the largest several prime factors of $n$.

The rest of the argument follows the proof of the main theorem in \cite{app}.

With $d_0$ chosen as above, we let $d$ be the unitary divisor of $n$ satisfying $\rad(d)=\rad(d_0)$. Write $n = de$. Our strategy is to count the number of values of $e$ corresponding to fixed $d,d_0$, then to sum on $d,d_0$. Since
\begin{equation}\label{eq:aexpression} a = \sigma(d)\sigma(e) - k de,\end{equation}
we have
\[ kde \equiv a \pmod{\sigma(d)}, \]
and thus $e$ is placed in a determined residue class modulo $\sigma(d)/\gcd(\sigma(d),kd)$. Since $e\le x/d$, the number of possibilities for $e$, given $d$ and $d_0$, is
\[ \ll \frac{x \cdot \gcd(\sigma(d),kd)}{d \sigma(d)} + 1 \ll_{k} \frac{x \gcd(\sigma(d),d)}{d\sigma(d)}+1 \ll \frac{x \gcd(\sigma(d),d)}{d^2} + 1. \]

Given $d_0$, the number of $d \le x$ with $\rad(d_0)=\rad(d)$ is $x^{o(1)}$ (see, e.g., \cite[Lemma 4.2]{pollack11}), and so when summing on $d,d_0$ the ``$+1$'' terms contribute at most $x^{3/5+o(1)}$. This is acceptable for us. We deal with the remaining terms as follows. Put $g=\gcd(d,\sigma(d))$; from \eqref{eq:aexpression}, we have that $g\mid a$. Thus,
\[ \sum_{d, d_0} \frac{x \gcd(\sigma(d),d)}{d^2} \le x \sum_{g\mid a} \sum_{\substack{d > x^{2/5} \\ g\mid d}} \frac{g}{d^2} \sum_{d_0 \mid d}1.  \]
The inner sum is just $\tau(d) \le x^{o(1)}$. Now writing $d=gh$, and summing on $h > x^{2/5}/g$, we see that the above expression is
\[ \le x^{1+o(1)} \sum_{g \mid a} \frac{1}{g} \sum_{h > x^{2/5}/g} \frac{1}{h^2} \ll x^{3/5+o(1)} \sum_{g \mid a} 1 \le x^{3/5+o(1)}.\]
Collecting all of our estimates completes the proof of the theorem.
\end{proof}

We close with the following result, which makes the upper bound in \cite[Corollary 3]{pom75} completely uniform.

\begin{proof}[Proof Sketch of Theorem \ref{prop:reallyfinalnotkidding}]
Let $a$ be an arbitrary integer.
We show there are at most $O(x/\log x)$ values of $n\le x$ with $P(n)\mid \sigma(n)-a$.  Via standard estimates,
we may assume that 
\begin{itemize}
\item[(a)] $x/\log x<n\le x$,
\item[(b)] $P(n)>x^{1/\log\log x}$,
\item[(c)] $n$ is not divisible by a proper power larger than $\log^2 x$.
\end{itemize}
Let $p=P(n)$.  By (b) and (c), we may assume that $p^2\nmid n$.  Write $n=pm$, so that
$\sigma(n)=\sigma(m)(p+1)$.  We have
\begin{equation}
\label{eq:mconstraint}
\sigma(m)\equiv a\pmod p.
\end{equation}
Assume that $p>x^{1/2}\log x$ and $x$ is large.  Then $\sigma(m)<m\log x<p$.  Thus, for a given $p$, all of the solutions
$m_1,m_2,\dots, m_t$ to \eqref{eq:mconstraint}, have $\sigma(m_1)=\sigma(m_2)=\dots=\sigma(m_t)$.
Now for any integer $c$, the number of solutions $m\le y$ to $\sigma(m)=c$ is $\le y^{1-(1+o(1))\log_3y/\log_2y}$
as $y\to\infty$, uniformly in $c$.  This is the $\sigma$-analogue of a result in \cite{Pom} (also see \cite{Pomtwo}) for $\varphi$.
So, there is an absolute constant $y_0$ such that if $y\ge y_0$, then the number of solutions $m\le y$ to
$\sigma(m)=c$ is $\le y^{1-1/\log_2y}$.  But for $O(x/\log x)$ choices for $n\le x$, we may assume
that $p\le x/y_0$.  Thus, the number of $m$ in \eqref{eq:mconstraint} for a given $p$ is at most
$(x/p)^{1-1/\log_2(x/p)}$.  We now sum this expression on $p$ with $x^{1/2}\log x<p\le x/y_0$.
If such a $p$ is in an interval $(x/e^{j+1},x/e^{j}]$, then the number of $m$ corresponding to $p$ is
at most $e^{j(1-1/\log j)}$.  Further, the number of choices for $p$ in this interval is $\ll x/(e^j\log(x/e^j))\ll x/(e^j\log x)$,
using $p>x^{1/2}\log x$.  Thus, the number of choices for $n$ is
$$
\ll \frac{x}{e^{j/\log j}\log x}.
$$
Summing on $j$ gets us $\ll x/\log x$.  

We now assume that $p\le x^{1/2}\log x$.  This implies that $m> x^{1/2}/\log^2 x$.  Let $m=uq$ where $q=P(m)$.
By standard estimates on smooth numbers, the number of integers $mp\le x$ with $m> x^{1/2}/\log^2 x$ and $P(m)\le \log x$ is at most $x^{1/2+o(1)}$ as
$x\to\infty$, so we may assume that $q>\log x$, and so, by (c), $q\nmid u$.  Thus, \eqref{eq:mconstraint}
implies that
\begin{equation}
\label{eq:qconstraint}
\sigma(u)q\equiv a-\sigma(u)\pmod p.
\end{equation}
We may assume that $p\nmid\sigma(n)$.  Indeed, otherwise, there is a prime power $r^a\mid n$
with $p\mid\sigma(r^a)$, and since $r^a>\frac12\sigma(r^a)\ge \frac12p$, (b) and (c) imply that $a=1$, that is,
$r\equiv-1\pmod p$.  Since we may assume that $x$ is large enough that $p>3$, we have $r>p$, contradicting $p=P(n)$.
Thus, we may assume that $p\nmid\sigma(u)$ in \eqref{eq:qconstraint}.  Hence, given $u,p$,
\eqref{eq:qconstraint} completely determines $q$, using $q<p$.  But the number of choices for
$u,p$ with $up<x/\log x$ is $<x/\log x$, and so this estimate completes the proof.
\end{proof}

\section*{Acknowledgments} The first author is supported by the National Science Foundation under Grant No. DMS-1402268. This work was completed while the second and third authors were in residence at the Mathematical Sciences Research Institute, during which time they were supported by the National Science Foundation under Grant No. DMS-1440140. The third author is also supported by an AMS Simons Travel Grant.



\end{document}